\documentclass[pdflatex,sn-basic]{sn-jnl}
\usepackage{graphicx}
\usepackage{multirow}
\usepackage{hyperref}
\usepackage{amssymb,amsfonts}
\usepackage{amsmath}
\usepackage{amsthm}
\usepackage{mathrsfs}
\usepackage[title]{appendix}
\usepackage{xcolor}
\usepackage{textcomp}
\usepackage{manyfoot}
\usepackage{booktabs}
\usepackage{algorithm}
\usepackage{algorithmicx}
\usepackage{algpseudocode}
\usepackage{listings}
\newtheorem{lemma}{Lemma}
\newtheorem{theorem}{Theorem}

\newtheorem{example}{Example}

\begin{document}

\title[Article Title]{Exponentially Fitted Finite Difference Approximation for Singularly Perturbed Fredholm Integro-Differential Equation}


\author[1]{\fnm{Mehebub} \sur{ Alam}}\email{mehebubalam7172@kgpian.iitkgp.ac.in}

\author*[1]{\fnm{Rajni Kant} \sur{ Pandey}}\email{rkp@maths.iitkgp.ac.in}


\affil[1]{\orgdiv{Department of Mathematics}, \orgname{Indian Institute of Technology}, \orgaddress{ \city{Kharagpur}, \postcode{721302}, \country{India}}}




\abstract{In this paper, we concentrate on solving second-order singularly perturbed Fredholm integro-differential equations (SPFIDEs). It is well known that solving these equations analytically is a challenging endeavor because of the presence of boundary and interior layers within the domain. To overcome these challenges, we develop a fitted second-order difference scheme that can capture the layer behavior of the solution accurately and efficiently, which is again, based on the integral identities with exponential basis functions, the composite trapezoidal rule, and an appropriate interpolating quadrature rules with the remainder terms in the integral form on a piecewise uniform mesh. Hence, our numerical method acts as a superior alternative to the existing methods in the literature. Further, using appropriate techniques in error analysis the scheme's convergence and stability have been studied in the discrete max norm. We have provided necessary experimental evidence that corroborates the theoretical results with a high degree of accuracy.}



\maketitle

\section{Introduction}
Fredholm integro-differential equations (FIDEs) play a significant role in various areas, such as mechanics, chemistry, electrostatics, physics, biology, fluid dynamics, astronomy, and so on \citep{kythe2002computational,polyanin2008handbook,cont2005integro}. Researchers have developed several theories, numerical calculations, and analyses for FIDEs, since these equations are crucial for the modeling of numerous phenomena in science and engineering. Various semi-analytical approaches, such as the Legendre polynomial approximation \citep{bildik2010comparison}, variational iteration method \citep{hamoud2019usage}, and differential transform method \citep{ziyaee2015differential} have been suggested in recent years. Furthermore, many numerical methods have been proposed in recent times. These include the Galerkin-Chebyshev wavelets method \citep{henka2022numerical}, the exponential spline method \citep{tahernezhad2020exponential}, the Nyström method \citep{tair2021solving}, the extrapolation method \citep{brezinski2019extrapolation}, and so on.  However, these studies have dealt with regular cases only.
\par Here we shall concern with the second-order FIDE of the form
  \begin{equation}
  \label{equation:1}
       \mathcal{L}_\epsilon 
 \mathit{v}:=\epsilon \mathit{v}''(\xi)+a(\xi) \mathit{v}'(\xi)=f(\xi)+\lambda \int_0^T K(\xi,\eta)\mathit{v}(\eta) \mathrm{d} \eta, \ \ \xi\in \Gamma :=(0,T),
  \end{equation}
\begin{equation}
\label{equation:2}
    \mathit{v}(0)=\alpha,\ \mathit{v}(T)=\beta,
\end{equation}
where, $0<\epsilon \leq 1,$ is a tiny parameter, $\lambda \in \mathbb{R}, $  $\Bar{\Gamma} = [0, T]$, and the functions $a(\xi) \geq \bar{a} > 0,$ where $\bar{a}$ is a number,  $f(\xi) \  (\xi \in \bar{\Gamma})$,\ $\  K(\xi, \eta)  \  ( (\xi, \eta) \in \bar{\Gamma} \times \bar{\Gamma}),\ \mathit{v}(\xi)\   (\xi \in \bar{\Gamma})$ are
sufficiently smooth. Singularly perturbed differential equations are special equations that usually involve a tiny number $\epsilon$ multiplying the highest order terms in the equations. When we solve these equations, we observe various phenomena happening at different scales. In certain narrow parts of the problem space, some derivatives change much faster than others. These narrow regions with rapid changes are called interior or boundary layers, depending on where they occur. Such equations are common in mathematical problems, for example: the study of moving air and how it affects structures, the behavior of fluids, how electricity behaves in complicated situations, different ways to understand how populations grow, creating models for neural networks, materials that remember their previous state, and mathematical models for how tiny particles move in a chaotic fluid \citep{qin2014integral,ahmad2008some,nieto2007new}. Usual discretization methods for solving problems with very small variations are known to be unstable and often do not provide good solutions when the variations are extremely small. Hence, there is a need to create consistent numerical approaches to tackle such problems.

Recent years have witnessed a substantial quantity of scholarly investigation devoted to the numerical solution of integro-differential equations that are singularly perturbed. Several robust difference methods for SPFIDEs have been suggested in the literature \citep{uniform_amiraliyev_2018,amiraliyev2020fitted,robust_durmaz_2021,durmaz2022numerical,durmaz2022parameter,cakir2022numerical,durmaz2023efficient,amiri2023effective}.
\cite{uniformly_iragi_2020}, \cite{second_mbroh_2020}, \cite{panda2021second}, and  \cite{liu2023novel} have introduced the numerous numerical approaches for singularly perturbed Volterra integro-differential equations (SPVIDEs).
SPVIDEs with delay have been investigated on uniform meshes in \citep{kudu2016finite,yapman2019convergence,amiraliyev2019fitted}. \cite{cakir2022exponentially,cakir2022new} have presented new difference schemes for first-order mixed Volterra-Fredholm integro-differential equations that are singularly perturbed. They also have proposed a novel and reliable difference scheme for solving second-order Volterra-Fredholm integro-differential equation with boundary layer \citep{cakir2022fitted}. Later, \cite{durmaz2023numerical} has explored a robust numerical approach for the same equations on a piecewise uniform mesh.
\par Although there is still a lack of extensive study on the numerical solution of SPFIDEs.  \cite{uniform_cimen_2021} developed a uniform numerical method with $O(\mathcal{N}^{-1})$ accuracy for the problem (\ref{equation:1})-(\ref{equation:2}) on a uniform mesh, where $\mathcal{N}$ is the mesh parameter. Our aim is to improve the accuracy of the method given in \citep{uniform_cimen_2021} on a non-uniform mesh utilizing the integral identities with the use of exponential basis functions and interpolating quadrature rules. This will capture the rapid variation near the boundary layers more accurately. 
\par This study is organized subsequently. In Section 2, we report some preliminary work that is relevant to the study. Section 3, proposes a difference scheme for SPFIDE. Later, in Section 4 we discuss the error analysis for the scheme. Then,  in Section 5, we present some numerical results that illustrate the scheme's performance. Finally, we conclude our work in Section 6.
\section{Preliminaries}

We shall use $C$ to represent a generic constant independent of the mesh parameter and $\epsilon$. The notation $\|\phi\|_{\infty}$ signifies the max norm for any continuous function $\phi(\xi)$ over the associated closed interval.\\
The development and convergence study of the appropriate numerical solution will be aided by the estimates given in the next lemma. These estimates will be used in the subsequent parts.
\begin{lemma} 
\label{Lemma:1}
Consider that $ f,\ a \in C^{2}(\bar{\Gamma}),\ K \in C^{2}(\bar{\Gamma} \times \bar{\Gamma})$ with $a(\xi) \geq \bar{a}>0$, and

$$
|\lambda |<\frac{\bar{a}}{\bar{K} T}
$$ 
where
$\bar{K}=\underset{\xi \in \bar{\Gamma}  }{\max} \int_0^T |K(\xi,\zeta)| \mathrm{d} \zeta,$
then, the solution $\mathit{v}$ of (\ref{equation:1})-(\ref{equation:2}) follows:

\begin{equation}
\label{equation:3}
\|\mathit{v}\|_{\infty} \leq C
\end{equation}

\begin{flalign}
\label{equation:4}
&\left|\mathit{v}^{\prime}(\xi)\right| \leq C\left(1+\frac{1}{\epsilon} e^{-\frac{\bar{a} \xi}{\epsilon}}\right),\\ \label{equation:5} 
&\left|\mathit{v}^{\prime \prime}(\xi)\right| \leq C\left(1+\frac{1}{\epsilon^2} e^{-\frac{\bar{a} \xi}{\epsilon}}\right).
\end{flalign}
\end{lemma}
\begin{proof}
To prove the first two inequalities (\ref{equation:3}) and (\ref{equation:4}), we follow a similar approach as the one in \cite{uniform_cimen_2021}. For the third one,  
differentiating equation (\ref{equation:1}), we obtain
\begin{equation}
\label{equation:6}
\epsilon U''(\xi)+ a(\xi) U'(\xi)= F(\xi),
\end{equation}
where $U'(\xi)=\mathit{v}''(\xi)$ and $F(\xi)=f'(\xi)+\lambda \int_0^T \frac{\partial K}{\partial \xi}\mathit{v}(\eta)\mathrm{d}\eta -a'(\xi)\mathit{v}'(\xi).$\\
From (\ref{equation:1}) and (\ref{equation:4}), we obtain:
\begin{equation}
\label{equation:7}
|\mathit{v}''(0)| \leq \frac{C}{\epsilon^2}.
\end{equation}
Integrating equation (\ref{equation:6}) from $0$ to $\xi$, we get:
\begin{equation}
\label{equation:8}
U'(\xi)=U'(0) e^{-Q(\xi)\epsilon^{-1}}+\epsilon^{-1} \int_0^\xi F(\zeta) e^{-(Q(\xi)-Q(\zeta))\epsilon^{-1}} \mathrm{d}\zeta,
\end{equation}
where
$$Q(\xi)=\int_0^\xi a(\tau) \mathrm{d}\tau.$$
Since 
$$\left|\epsilon^{-1}\int_0^\xi F(\zeta) e^{-(Q(\xi)-Q(\zeta))\epsilon^{-1}} \mathrm{d}\zeta\right| \leq C.$$
From (\ref{equation:8}), we have
\begin{equation} \label{equation:9}
\mathit{v}''(\xi)=C \epsilon^{-2} e^{-\epsilon \bar{a} \xi}+C,
\end{equation}
immediately leads to (\ref{equation:5}).

\end{proof}

 \section{Proposed difference scheme}  
Consider $\Gamma_\mathcal{N}$ as a non-uniform mesh on $\Gamma$: \begin{equation*}
\Gamma_\mathcal{N}= \{ 0<\xi_1<\xi_2<\ldots<\xi_{\mathcal{N}-1}, h_i=\xi_i-\xi_{i-1}\}
\end{equation*}
and 
\begin{equation*}
    \bar{\Gamma }_\mathcal{N}=\Gamma_\mathcal{N} \cup \{\xi_0=0,\ \xi_\mathcal{N}=T\}.
\end{equation*}
We use the following difference rules to define a mesh function $\phi(\xi)$ on the mesh $\Gamma_\mathcal{N}$:
\begin{equation*}
    \phi_i=\phi(\xi_i),\hspace{0.3cm}  \hbar_i=\frac{h_i+h_{i+1}}{2},\  i=1,2,\ldots \mathcal{N}-1,\   \hbar_{0}=\frac{h_{1}}{2}, \hbar_{\mathcal{N}}=\frac{h_{\mathcal{N}}}{2},\hspace{.3cm}   \|\phi\|_{\infty, \Bar{\Gamma}_\mathcal{N}}:=\max_{1 \leq i \leq \mathcal{N}} |\phi_i|,
\end{equation*}
\begin{equation*}
\phi_{\bar{\xi},i}=\frac{\phi_i-\phi_{i-1}}{h_i}, \hspace{1cm} \phi_{\xi,i}=\frac{\phi_{i+1}-\phi_i}{h_{i+1}}, \hspace{1cm} \phi_{\Bar{\xi} \xi,i}=\frac{\phi_{\xi,i}-\phi_{\Bar{\xi},i}}{\hbar_{i}}, \hspace{1cm} \phi_{\stackrel{0}{\xi},i}=\frac{\phi_{i+1}-\phi_{i-1}}{2\hbar_{i}} .
\end{equation*}
We establish the difference scheme on $\Gamma_\mathcal{N}$ for the problem (\ref{equation:1})-(\ref{equation:2}). We split each of the subintervals $[0, \rho]$ and $[\rho, T]$ for an even number $\mathcal{N}$ into $\frac{\mathcal{N}}{2}$ equidistant sub intervals. $\rho$, the transition point that separates fine and coarse portions of the mesh is presented as
$$\rho = \text{min}\{\frac{T}{2}
,\ \bar{a}^{-1}\epsilon \ln \mathcal{N}\}.$$\\
We use $h^{(1)}$ for step length in $[0,\rho]$ and $h^{(2)}$ for step length in $[\rho,T]$. Thus, mesh step sizes hold
$$h^{(1)} = \frac{2\rho}{\mathcal{N}},\  h^{(2)} =\frac{2(T- \rho)}{\mathcal{N}}
.$$
\\
The mesh points of $\Gamma_\mathcal{N}$ are  identified as:
$$ \xi_i=\begin{cases}
ih^{(1)}& \text{if}\  
 i=0,1,\ldots,\frac{\mathcal{N}}{2}\\
\left(i-\frac{\mathcal{N}}{2}\right)h^{(2)}+\rho & \text{if}\   i=\frac{\mathcal{N}}{2}+1,\ldots,\mathcal{N}.
\end{cases}$$
Let us begin by considering the integral identity for the equation (\ref{equation:1}):
\begin{equation}
\label{equation:10}
\hbar^{-1}_i\int_{\xi_{i-1}}^{\xi_{i+1}}\left[\mathcal{L}_\epsilon \mathit{v}(\xi)\right]\psi_i(\xi)\mathrm{d}\xi= \hbar^{-1}_i \int_{\xi_{i-1}}^{\xi_{i+1}} \left[f(\xi)+ \lambda \int_0^T K(\xi,\eta)\mathit{v}(\eta)\right]\psi_i(\xi)\mathrm{d}\xi,\ i=1,2,\ldots,\mathcal{N}-1,
\end{equation}
using the basis function is specified as:
\begin{equation}
\label{equation:11}
\psi_i(\xi)=\begin{cases} 
      \psi_i^{(1)}(\xi)=\frac{e^{\frac{a_i(\xi-\xi_{i-1})}{\epsilon}}-1}{e^{\frac{a_ih_i}{\epsilon}}-1},&\xi\in(\xi_{i-1},\xi_{i}) \\
      \psi_i^{(2)}(\xi)=\frac{1-e^{-\frac{a_i(\xi_{i+1}-\xi)}{\epsilon}}}{1-e^{-\frac{a_ih_{i+1}}{\epsilon}}},&\xi\in(\xi_{i},\xi_{i+1})\\
      0 ,& \xi \notin(\xi_{i-1},\xi_{i+1}), 
   \end{cases}
\end{equation}
where $\psi^{(1)}_i(\xi)$  and $\psi^{(2)}_i(\xi)$ are the solutions to the following equations:

\begin{align*}
&\epsilon \psi_i^{''} (\xi) -a_i \psi_i^{'} (\xi)=0,\ \  \xi_{i-1}<\xi<\xi_i,\\
&\psi_i(\xi_{i-1})= 0,\ \ \psi_i(\xi_{i})= 1.
\end{align*}
and
\begin{align*}
&\epsilon \psi_i^{''} (\xi) -a_i \psi_i^{'} (\xi)=0,\ \  \xi_{i-1}<\xi<\xi_i,\\
&\psi_i(\xi_{i})= 1,\ \ \psi_i(\xi_{i+1})= 0,
\end{align*}
respectively.
For the difference part from (\ref{equation:10}), we obtain the following by using the appropriate interpolating quadrature formulas (see, e.g., \cite{amiraliyev1995difference}):
\begin{flalign}
\label{equation:12}
\nonumber
\hbar^{-1}_i \int_{\xi_{i-1}}^{\xi_{i+1}} [\mathcal{L}_\epsilon\mathit{v}]\psi_i(\xi)\mathrm{d}\xi&=\hbar^{-1}_i \left[\int_{\xi_{i-1}}^{\xi_{i+1}}\epsilon \mathit{v}^{''}\psi_i(\xi)\mathrm{d}\xi+\int_{\xi_{i-1}}^{\xi_{i+1}}a(\xi)\mathit{v}^{'}(\xi)\psi_i(\xi)\mathrm{d}\xi\right]\\
& =\epsilon \mathit{v}_{\xi\bar{\xi},i}+a_i\left(\chi_i^{(1)}\mathit{v}_{\bar{\xi},i}+\chi_i^{(2)} \mathit{v}_{\xi,i}\right)+\hbar^{-1}_i \int_{\xi_{i-1}}^{\xi_{i+1}} \left(a(\xi)-a(\xi_i)\right)\mathit{v}(\xi)\psi_i(\xi)\mathrm{d}\xi,
\end{flalign}
where
\begin{equation}
\label{equation:13}
\chi^{(1)}_i=\hbar^{-1}_i\left[\frac{\epsilon}{a_i}-\frac{h_i}{e^{\frac{a_ih_i}{\epsilon}}-1} \right],
\end{equation}
and
\begin{equation}
\label{equation:14}
\chi^{(2)}_i=\hbar^{-1}_i\left[\frac{h_{i+1}}{1-e^{\frac{a_ih_{i+1}}{\epsilon}}}-\frac{\epsilon}{a_i} \right].
\end{equation}
With respect to the mesh points $\xi_i,\xi_{i+1},$ the Newton interpolation formula yields:
\begin{equation*}
 a(\xi)=a(\xi_i)+(\xi-\xi_i)a_{\xi,i}+\frac{1}{2}(\xi-\xi_i)(\xi-\xi_{i+1})a^{''}(\tau_i(\xi)).   
\end{equation*}
Consequently, we derive: 
\begin{flalign}
  \hbar^{-1}_i \int_{\xi_{i-1}}^{\xi_{i+1}} \left(a(\xi)-a(\xi_i)\right)\mathit{v}'(\xi)\psi_i(\xi)\mathrm{d}\xi =&
\hbar^{-1}_ia_{\xi,i} \int_{\xi_{i-1}}^{\xi_{i+1}} (\xi-\xi_i)\mathit{v}'\psi_i(\xi)\mathrm{d}\xi+ \nonumber\\&
\frac{1}{2} \hbar^{-1}_i \int_{\xi_{i-1}}^{\xi_{i+1}} (\xi-\xi_i)(\xi-\xi_{i+1})a''(\tau_i(\xi))\mathit{v}'(\xi)\psi_i(\xi)\mathrm{d}\xi\nonumber\\
=&   a_{\xi,i}\left( \mathit{v}_{\Bar{\xi},i} \gamma_i^{(1)}+\mathit{v}_{\xi,i} \gamma_i^{(2)} \right) +\mathrm{R}_i^{(1)} ,\label{equation:15}
\end{flalign}
where
\begin{flalign}
\gamma_i^{(1)}=&\hbar^{-1}_i \left[ \frac{1}{e^{\frac{a_ih_i}{\epsilon}}-1} \left\{ h_i \left( \frac{\epsilon}{a_i}-\xi_{i-1} \right)+\left( \frac{\xi_i^2}{2} - \frac{\xi_{i-1}^2}{2} \right) \right\} - \frac{\epsilon^2}{a_i^2} \right], \label{equation:16}\\
\gamma_i^{(2)}=&\hbar^{-1}_i \left[ \frac{1}{1-e^{-\frac{a_ih_{i+1}}{\epsilon}}} \left\{ h_{i+1} \left(\xi_{i+1} -\frac{\epsilon}{a_i} \right)-\left( \frac{\xi_{i+1}^2}{2} - \frac{\xi_{i}^2}{2} \right) \right\} + \frac{\epsilon^2}{a_i^2} \right], \label{equation:17}
\end{flalign}
and 
\begin{equation}
\label{equation:18}
\mathrm{R}_i^{(1)}=\frac{1}{2} \hbar^{-1}_i \int_{\xi_{i-1}}^{\xi_{i+1}} (\xi-\xi_i)(\xi-\xi_{i+1})a''(\tau_i(\xi))\mathit{v}'(\xi)\psi_i(\xi)\mathrm{d}\xi.
\end{equation}
Thereby, the identity (\ref{equation:12}) reduces to:
\begin{equation}
\label{equation:19}
\hbar^{-1}_i \int_{\xi_{i-1}}^{\xi_{i+1}} [\mathcal{L}_\epsilon \mathit{v}]\psi_i(\xi)\mathrm{d}\xi=\epsilon \mathit{v}_{\Bar{\xi} \xi,i}+\hat{a}_i^{(1)}\mathit{v}_{\bar{\xi},i}+\hat{a}_i^{(2)}\mathit{v}_{\xi,i} +\mathrm{R}_i^{(1)},
\end{equation}
where
\begin{equation}
\label{equation:20}
\hat{a}_i^{(1)}= a_i \chi_i^{(1)}+a_{\xi,i}\gamma_i^{(1)},
\end{equation}
\begin{equation}
\label{equation:21}
\hat{a}_i^{(2)}= a_i \chi_i^{(2)}+a_{\xi,i}\gamma_i^{(2)}.
\end{equation}
Upon substituting
$$\mathit{v}_{\Bar{\xi},i}=\mathit{v}_{\stackrel{0}{\xi},i}-\frac{h_{i+1}}{2}\mathit{v}_{\Bar{\xi}\xi,i},\ \ \mathit{v}_{\xi,i}=\mathit{v}_{\stackrel{0}{\xi},i}+\frac{h_{i}}{2}\mathit{v}_{\Bar{\xi}\xi,i}
$$
into (\ref{equation:19}), finally we get:
\begin{equation}
\label{equation:22}
\hbar^{-1}_i \int_{\xi_{i-1}}^{\xi_{i+1}} [\mathcal{L}_\epsilon \mathit{v}]\psi_i(\xi)\mathrm{d}\xi=\epsilon \theta_i \mathit{v}_{\Bar{\xi}\xi,i}+ A_i
\mathit{v}_{\stackrel{0}{\xi},i} +R_i^{(1)},
\end{equation}
where
\begin{equation}
\label{equation:23}
\theta_i= 1-\frac{\hat{a}_i^{(1)} h_{i+1}}{2\epsilon}+ \frac{\hat{a}_i^{(2)} h_{i}}{2\epsilon}\  \text{and}\  A_i=\hat{a}_i^{(1)}+\hat{a}_i^{(2)}.
\end{equation}
Similarly, we derive:

\begin{equation}
\label{equation:24}
\hbar_{i}^{-1} \int_{\xi_{i-1}}^{\xi_{i+1}} f(\xi) \psi_{i}(\xi) \mathrm{d} \xi= \hat{f}_i +\mathrm{R}_{i}^{(2)},
\end{equation}
where
\begin{equation}
\label{equation:25}
\hat{f}_i=f_{i}\left( \chi_i^{(1)}+\chi_i^{(2)}\right)+f_{\xi, i} \left( \gamma_i^{(1)}+\gamma_i^{(2)}\right)
\end{equation}
and
\begin{equation}
\label{equation:26}
\mathrm{R}_{i}^{(2)}  =\frac{1}{2}   \hbar_{i}^{-1} \int_{\xi_{i-1}}^{\xi_{i+1}}\left(\xi-\xi_{i}\right)\left(\xi-\xi_{i+1}\right) f^{\prime \prime}\left(\tau_{i}(\xi)\right) \psi_{i}(\xi) \mathrm{d} \xi .
\end{equation}
The approximation for the integral term of the right-hand side of (\ref{equation:10}) still has to be obtained.
Applying the Taylor expansion:

$$
K(\xi, \eta)=K\left(\xi_{i}, \eta\right)+\left(\xi-\xi_{i}\right) \frac{\partial}{\partial \xi} K\left(\xi_{i}, \eta\right)+\frac{\left(\xi-\xi_{i}\right)^{2}}{2} \frac{\partial^{2}}{\partial \xi^{2}} K\left(\tau_{i}(\xi), \eta\right),
$$
we get:
 \begin{flalign}
\hbar_{i}^{-1} \lambda &\int_{\xi_{i-1}}^{\xi_{i+1}} \psi_{i}(\xi)\left(\int_{0}^{T} K(\xi, \eta) \mathit{v}(\eta) \mathrm{d} \eta\right) \mathrm{d} \xi=\lambda \left( \chi_i^{(1)}+\chi_i^{(2)}\right) \int_{0}^{T} K \left(\xi_{i}, \eta\right) \mathit{v}(\eta) \mathrm{d} \eta \nonumber\\
&+ \lambda \left( \gamma_i^{(1)}+\gamma_i^{(2)}\right) \int_{0}^{T} \frac{\partial}{\partial \xi} K\left(\xi_{i}, \eta\right) \mathit{v}(\eta) \mathrm{d} \eta+\mathrm{R}_{i}^{(3)} \equiv \lambda \int_{0}^{T} \mathcal{K}\left(\xi_{i}, \eta\right) \mathit{v}(\eta) \mathrm{d} \eta+\mathrm{R}_{i}^{(3)},\label{equation:27}
\end{flalign}   
where

\begin{equation}
\label{equation:28}
\mathcal{K}\left(\xi_{i}, \eta\right)=K\left(\xi_{i}, \eta\right) \left( \chi_i^{(1)}+\chi_i^{(2)}\right) +  \frac{\partial}{\partial \xi} K \left(\xi_{i}, \eta\right) \left( \gamma_i^{(1)}+\gamma_i^{(2)}\right), 
\end{equation}

\begin{equation}
\label{equation:29}
\mathrm{R}_{i}^{(3)}=\frac{1}{2} \hbar_{i}^{-1} \int_{\xi_{i-1}}^{\xi_{i+1}}\left(\xi-\xi_{i}\right)^{2} \psi_{i}(\xi)\left(\int_{0}^{T} \frac{\partial^{2}}{\partial \xi^{2}} K \left(\tau_{i}(\xi), \eta\right) \mathit{v}(\eta) \mathrm{d} \eta\right) \mathrm{d} \xi .
\end{equation}
Now we require the composite trapezoidal rule on $[0, T]$ with integral remainder term:

\begin{equation}
\label{equation:30}
\int_{0}^{T} F(\eta) \mathrm{d} \eta=\sum_{i=0}^{\mathcal{N}} \hbar_{i} F_{i}+\mathrm{R}_{\mathcal{N}}
\end{equation}
and
$$
\mathrm{R}_{\mathcal{N}}=\frac{1}{2} \sum_{i=1}^{\mathcal{N}} \int_{\xi_{i-1}}^{\xi_{i}}\left(\tau-\xi_{i}\right)\left(\tau-\xi_{i-1}\right) F^{\prime \prime}(\tau) \mathrm{d} \tau.
$$
Next, apply the formula (\ref{equation:30}) on $[0, T]$ to calculate $\int_0^{T} \mathcal{K}(\xi_i, \eta)\mathit{v}(\eta)\mathrm{d} \eta$:
\begin{equation}
\label{equation:31}
\int_{0}^{T} \mathcal{K} \left(\xi_{i}, \eta\right) \mathit{v}(\eta) \mathrm{d} \eta=\sum_{j=0}^{\mathcal{N}} \hbar_{j} \mathcal{K}_{i j} \mathit{v}_{j}+\mathrm{R}_{i}^{(4)},
\end{equation}
where
\begin{equation}
\label{equation:32}
\mathrm{R}_{i}^{(4)}=\frac{1}{2} \sum_{j=1}^{\mathcal{N}} \int_{\xi_{j-1}}^{\xi_{j}}\left(\tau-\xi_{j}\right)\left(\tau-\xi_{j-1}\right) \frac{\mathrm{d}^{2}}{\mathrm{d} \tau^{2}}\left(\mathcal{K}\left(\xi_{i}, \tau\right) \mathit{v}(\tau)\right) \mathrm{d} \tau .
\end{equation}
Therefore the relation (\ref{equation:27}) reduces to:
\begin{equation}
\label{equation:33}
\hbar_{i}^{-1} \lambda \int_{\xi_{i-1}}^{\xi_{i+1}} \psi_{i}(\xi)\left(\int_{0}^{T} K(\xi, \eta) \mathit{v}(\eta) \mathrm{d} \eta\right) \mathrm{d} \xi =\lambda \sum_{j=0}^{\mathcal{N}} \hbar_{j} \mathcal{K}_{i j} \mathit{v}_{j}+\mathrm{R}_{i}^{(3)}+\mathrm{R}_{i}^{(4)}.
\end{equation}
By considering (\ref{equation:22}), (\ref{equation:24}), and (\ref{equation:33}) in (\ref{equation:10}), we derive the discrete identity for $\mathit{v}(\xi)$:

\begin{equation}
\label{equation:34}
 \epsilon \theta_i \mathit{v}_{\Bar{\xi} \xi,i}+ A_i
\mathit{v}_{\stackrel{0}{\xi},i}= \hat{f}_{i}+ \lambda \sum_{j=0}^{\mathcal{N}} \hbar_{j} \mathcal{K}_{i j} \mathit{v}_{j}+\mathrm{R}_i , \ 1\leq i \leq \mathcal{N}-1,
\end{equation}
with the remainder term:
\begin{equation}
\label{equation:35}
 \mathrm{R}_{i}= -\mathrm{R}_{i}^{(1)}+\mathrm{R}_{i}^{(2)}+\mathrm{R}_{i}^{(3)}+\mathrm{R}_{i}^{(4)},   
\end{equation}
where $\mathrm{R}_{i}^{(k)},(k=1,2,3,4)$ are defined by (\ref{equation:18}), (\ref{equation:26}), (\ref{equation:29}), and (\ref{equation:32}), respectively.
Neglecting $\mathrm{R}_i$ in (\ref{equation:34}) yields the difference scheme for solving the problem (\ref{equation:1})–(\ref{equation:2}):

\begin{equation}
\label{equation:36}
 \epsilon \theta_i y_{\Bar{\xi} \xi,i}+ A_i
y_{\stackrel{0}{\xi},i} =\hat{f}_{i}+ \lambda \sum_{j=0}^{\mathcal{N}} \hbar_{j} \mathcal{K}_{i j} y_{j}, \ 1\leq i \leq \mathcal{N}-1,
\end{equation}
\begin{equation}
\label{equation:37}
y_{0}=\alpha, \quad y_{\mathcal{N}}=\beta,   
\end{equation}
where $\theta_i,$  $A_i$, $\hat{f}_i $, and $\mathcal{K}_{i j}$ are given by (\ref{equation:23}),  (\ref{equation:23}), (\ref{equation:25}), and (\ref{equation:28}) respectively.
\section{Error analysis}
This section evaluates the proposed method's convergence. Suppose $z_i = y_i-\mathit{v}_i$ represents the error in the difference scheme (\ref{equation:36})-(\ref{equation:37}). Equations (\ref{equation:34}) and (\ref{equation:36}) provide $z_i$ as the discrete problem solution:
\begin{equation}
\label{equation:38}
 \epsilon \theta_i z_{\Bar{\xi} \xi,i}+ A_i
z_{\stackrel{0}{\xi},i} =-\mathrm{R}_{i}+ \lambda \sum_{j=0}^{\mathcal{N}} \hbar_{j} \mathcal{K}_{i j} z_{j}, \ 1\leq i \leq \mathcal{N}-1,
\end{equation}
\begin{equation}
\label{equation:39}
z_{0}=0, \quad z_{\mathcal{N}}=0,   
\end{equation}
\begin{lemma} \label{Lemma:2}
If the conditions of Lemma \ref{Lemma:1} hold, then the error
function $\mathrm{R}_i$ satisfies the estimate:
\begin{equation} \label{equation:40}
||\mathrm{R}||_{\infty,\Gamma_\mathcal{N}} \leq C \mathcal{N}^{-2} \ln \mathcal{N}.
\end{equation}
\end{lemma}
\begin{proof}
Firstly, we estimate $\mathrm{R}_i^{(2)}$, since $a \in C^2[0,T],\  |\xi-\xi_i| \leq \max ( h_i, h_{i+1} ),\  |\xi-\xi_{i+1}| \leq 2 \max ( h_i, h_{i+1} ),\  0 < \psi_i(\xi) \leq 1,\  \text{and}\   h^{(k)} \leq C\mathcal{N}^{-1},\  (k=1,2).$ From the equation (\ref{equation:26}), we then have: 
\begin{align}
  \left|\mathrm{R}_i^{(2)}\right| & \leq 
C \hbar^{-1}_i \int_{\xi_{i-1}}^{\xi_{i+1}} |(\xi-\xi_i)(\xi-\xi_{i+1})| \psi_i(\xi)\mathrm{d}\xi  \nonumber \\ & \leq 
C \left\{ \max ( h_i, h_{i+1} ) \right\} ^2 \nonumber \\& \leq C\mathcal{N}^{-2}.
\label{equation:41}
\end{align}
For $\mathrm{R}_i^{(1)}$, we get:
\begin{align}
  \left|\mathrm{R}_i^{(1)}\right| & \leq 
C \hbar^{-1}_i \int_{\xi_{i-1}}^{\xi_{i+1}} |(\xi-\xi_i)(\xi-\xi_{i+1}) \psi_i(\xi) \mathit{v}'(\xi)| \mathrm{d}\xi  \nonumber \\& \leq  
C \mathcal{N}^{-1}\int_{\xi_{i-1}}^{\xi_{i+1}} |\mathit{v}'(\xi)| \mathrm{d}\xi \nonumber \\& \leq  C\mathcal{N}^{-1} \int_{\xi_{i-1}}^{\xi_{i+1}} \frac{1}{\epsilon} e^{-\frac{\bar{a} \xi}{\epsilon}} \mathrm{d}\xi.
\label{equation:42}
\end{align}
Now, when $\rho=\frac{T}{2},$ we have $\frac{T}{2} \leq \bar{a}^{-1} \epsilon \ln \mathcal{N} \ \text{and}\  h^{(1)}=h^{(2)}=h=\frac{T}{2}$, then
$$\int_{\xi_{i-1}}^{\xi_{i+1}} \frac{1}{\epsilon} e^{-\frac{\bar{a} \xi}{\epsilon}} \mathrm{d}\xi \leq \frac{2h}{\epsilon} \leq \frac{4 \bar{a}^{-1}}{T} \mathcal{N}^{-1} \ln \mathcal{N} $$
Therefore (\ref{equation:42}) implies 
$$ \left|\mathrm{R}_i^{(1)}\right| \leq C \mathcal{N}^{-2} \ln \mathcal{N},\ \ \text{when}\ \  \rho=\frac{T}{2} $$  
and while $\rho=\bar{a}^{-1} \epsilon \ln \mathcal{N} \leq \frac{T}{2}$, then
$$\int_{\xi_{i-1}}^{\xi_{i+1}} \frac{1}{\epsilon} e^{-\frac{\bar{a} \xi}{\epsilon}} \mathrm{d}\xi \leq \frac{2|h|}{\epsilon} \leq C \mathcal{N}^{-1}  \ln \mathcal{N},\ \ 1 \leq i \leq \frac{\mathcal{N}}{2}$$
and\begin{small}

$$\int_{\xi_{i-1}}^{\xi_{i+1}} \frac{1}{\epsilon} e^{-\frac{\bar{a} \xi}{\epsilon}} \mathrm{d}\xi= \bar{a}^{-1} e^{-\frac{\bar{a} \xi_{i-1}}{\epsilon}} \left( 1-e^{-\frac{\bar{a} }{\epsilon}(\xi_{i+1}-\xi_{i-1}) } \right) \leq  C\mathcal{N}^{-1},\ \ \frac{\mathcal{N}}{2}+1\leq i \leq \mathcal{N} $$ 
\end{small}
Thereby, 
\begin{equation}
\label{equation:43}
\left|\mathrm{R}_i^{(1)}\right| \leq C \mathcal{N}^{-2} \ln \mathcal{N},\ \  \text{when}\ \ \rho=\bar{a}^{-1} \epsilon \ln \mathcal{N}. 
\end{equation}
Thirdly, for $\mathrm{R}_i^{(3)}$, the boundness of $\frac{\partial^{2}K}{\partial \xi^{2}}$ and $\|\mathit{v}\|_\infty \leq C$, we get
\begin{align}
  \left|\mathrm{R}_i^{(3)}\right| & \leq 
C \hbar^{-1}_i \int_{\xi_{i-1}}^{\xi_{i+1}} |\xi-\xi_i|^2 |\psi_i(\xi)|\mathrm{d}\xi  \nonumber \\ & \leq 
C \left\{ \max ( h_i, h_{i+1} ) \right\} ^2 \nonumber \\& \leq C\mathcal{N}^{-2}.
\label{equation:44}
\end{align}
Lastly, for the estimation of $\mathrm{R}_i^{(4)}$, we have:
\begin{align}
  \left|\mathrm{R}_i^{(4)}\right| & \leq 
\left| \frac{1}{2} \sum_{j=1}^{\mathcal{N}} \int_{\xi_{j-1}}^{\xi_{j}}\left(\tau-\xi_{j}\right)\left(\tau-\xi_{j-1}\right) \frac{\mathrm{d}^{2}}{\mathrm{d} \tau^{2}}\left(\mathcal{K}\left(\xi_{i}, \tau\right) \mathit{v}(\tau)\right) \mathrm{d} \tau\right| \nonumber \\ & \leq C \sum_{j=1}^{\mathcal{N}} \int_{\xi_{j-1}}^{\xi_{j}}\left(\xi_{j}-\tau\right)\left(\tau-\xi_{j-1}\right) (1+|\mathit{v}'(\xi)|+|\mathit{v}''(\xi)|) \mathrm{d} \tau \nonumber \\ & \leq C\left(\sum_{j=1}^{\mathcal{N}} h_j^{3}+ \sum_{j=1}^{\mathcal{N}} \int_{\xi_{j-1}}^{\xi_{j}}\left(\xi_{j}-\tau\right)\left(\tau-\xi_{j-1}\right) (|\mathit{v}'(\xi)|+|\mathit{v}''(\xi)|) \mathrm{d} \tau  \right) \nonumber \\ & \leq C\left(\sum_{j=1}^{\mathcal{N}} h_j^{3}+ \sum_{j=1}^{\mathcal{N}} \int_{\xi_{j-1}}^{\xi_{j}}\left(\xi_{j}-\tau\right)\left(\tau-\xi_{j-1}\right)\frac{1}{\epsilon^2}e^{-\frac{\bar{a} \tau}{\epsilon}}  \mathrm{d} \tau  \right).
\label{equation:45}
\end{align}
We can simplify the first term on the right side of the inequality (\ref{equation:45}) in the following way:
\begin{equation}
\label{equation:46}
\sum_{j=1}^{\mathcal{N}} h_j^{3}=\sum_{j=1}^{\frac{\mathcal{N}}{2}} \left( h^{(1)} \right)^3+\sum_{\frac{\mathcal{N}}{2}}^\mathcal{N} \left( h^{(2)} \right)^3 \leq C\mathcal{N}^{-2},
\end{equation}
and for the rest of the right side of the inequality (\ref{equation:45})  can be expressed as:
\begin{align}
 &\sum_{j=1}^{\frac{\mathcal{N}}{2}} \int_{\xi_{j-1}}^{\xi_{j}}\left(\xi_{j}-\tau\right)\left(\tau-\xi_{j-1}\right)\frac{1}{\epsilon^2}e^{-\frac{\bar{a} \tau}{\epsilon}}  \mathrm{d} \tau 
 \nonumber \\ &\leq \left| h^{(1)} \right|^2 \int^\rho_0 \frac{1}{\epsilon^2}e^{-\frac{\bar{a} \tau}{\epsilon}}  \mathrm{d} \tau \nonumber \\ & \leq  \left| h^{(1)} \right|^2 \bar{a}^{-1} \epsilon^{-1},
 \label{equation:47}
\end{align} 
now if $\rho = \bar{a}^{-1} \epsilon \ln \mathcal{N} < \frac{T}{2},$ then $ \left| h^{(1)} \right|^2 \bar{a}^{-1} \epsilon^{-1} < 2T\bar{a}^{-2}\mathcal{N}^{-2} \ln \mathcal{N},$
if $\rho = \frac{T}{2}$ then  $\left| h^{(1)} \right|^2 \bar{a}^{-1} \epsilon^{-1} < 2T\bar{a}^{-1}\mathcal{N}^{-2} \ln \mathcal{N},$ 
and
\begin{align}
 &\sum_{\frac{\mathcal{N}}{2}}^\mathcal{N} \int_{\xi_{j-1}}^{\xi_{j}}\left(\xi_{j}-\tau\right)\left(\tau-\xi_{j-1}\right)\frac{1}{\epsilon^2}e^{-\frac{\bar{a} \tau}{\epsilon}}  \mathrm{d} \tau 
 \nonumber \\ &\leq 2 \bar{a}^{-1} \sum_{\frac{\mathcal{N}}{2}}^\mathcal{N} \int_{\xi_{j-1}}^{\xi_{j}}\left(\xi_{j}-\tau-\frac{h^{(2)}}{2}\right) \frac{1}{\epsilon}e^{-\frac{\bar{a} \tau}{\epsilon}}  \mathrm{d} \tau \nonumber \\ & \leq 2\bar{a}^{-1} h^{(2)} \int^T_\rho \frac{1}{\epsilon}e^{-\frac{\bar{a} \tau}{\epsilon}}  \mathrm{d} \tau \nonumber \\ & = 2\bar{a}^{-2} h^{(2)} \left(e^{-\frac{\bar{a} \rho}{\epsilon}}-e^{-\frac{\bar{a} T}{\epsilon}} \right) \nonumber \\ & \leq 2\bar{a}^{-2} h^{(2)}\mathcal{N}^{-1} \nonumber \\ & \leq C\mathcal{N}^{-2}
 \label{equation:48}.
\end{align} 
The inequalities (\ref{equation:47}) and(\ref{equation:48}) implies 
\begin{equation}
 \sum_{j=1}^\mathcal{N}\int_{\xi_{j-1}}^{\xi_{j}}\left(\xi_{j}-\tau\right)\left(\tau-\xi_{j-1}\right)\frac{1}{\epsilon^2}e^{-\frac{\bar{a} \tau}{\epsilon}}  \mathrm{d} \tau \leq C\mathcal{N}^{-2}\ln \mathcal{N}, \forall \  i=1,2,\ldots,\mathcal{N}
 \label{equation:49} 
\end{equation}
based on (\ref{equation:46})-(\ref{equation:49}), from (\ref{equation:45}), we estimate the following inequality:
\begin{equation}\label{equation:50}
\left|\mathrm{R}_i^{(4)}\right| \leq C \mathcal{N}^{-2} \ln \mathcal{N}, \ \forall \  i=1,2,\ldots,\mathcal{N}. 
\end{equation}
 Therefore the inequality (\ref{equation:41}) together with (\ref{equation:43}),(\ref{equation:44}), and (\ref{equation:50}), we arrive at (\ref{equation:40}).
\end{proof}
Since $a(\eta)\geq \bar{a} > 0$, it follows that $A_i \geq \bar{A} > 0$, $\bar{A}$ is a number,  when $\mathcal{N}$ is large enough.
\begin{lemma} \label{Lemma:3}
Suppose that the conditions of Lemma \ref{Lemma:1} are satisfied and if  
$$|\lambda | < \frac{\bar{A}}{T\widetilde{K} },$$
where
$\widetilde{K}=\underset{1\leq i\leq \mathcal{N}}{\max} \sum_{j=1}^\mathcal{N} \hbar_j |\mathcal{K}_{ij} |, $ then the solution of the problem (\ref{equation:38}) and (\ref{equation:39}) satisfies:
\begin{equation}\label{equation:51}
 \|z\|_{\infty, \Bar{\Gamma}_\mathcal{N}} \leq C\|\mathrm{R}\|_{\infty, \Bar{\Gamma}_\mathcal{N}}.   
\end{equation}
\end{lemma}
\begin{proof}
Here, we employ discrete Green's function $\mathcal{G}^{h}\left(\xi_{i}, \zeta_{j}\right)$ for the operator:
$$
\begin{aligned}
& \mathcal{L}_\epsilon^{\mathcal{N}} z_{i}:=-\epsilon \theta_{i} z_{\bar{\xi} \xi, i}-A_{i} z_{0, i},\  1 \leq i \leq \mathcal{N}-1, \\
& z_{0}=z_{\mathcal{N}}=0 .
\end{aligned}
$$
Namely, the $\mathcal{G}^{h}\left(\xi_{i}, \zeta_{j}\right)$ is expressed as a function of $\xi_{i}$ for fixed $\zeta_{j}$ :
$$
\begin{aligned}
&\mathcal{L}_\epsilon^{\mathcal{N}} \mathcal{G}^{h}\left(\xi_{i}, \zeta_{j}\right)=\delta^{h}\left(\xi_{i}, \zeta_{j}\right), \xi_{i} \in \Gamma_{\mathcal{N}}, \zeta_{j} \in \Gamma_{\mathcal{N}} \\
& \mathcal{G}^{h}\left(0, \zeta_{j}\right)=\mathcal{G}^{h}\left(T, \zeta_{j}\right)=0, \zeta_{j} \in \Gamma_{\mathcal{N}},
\end{aligned}
$$
where the Kronecker delta is represented as $\delta^{h}\left(\xi_{i}, \zeta_{j}\right)=\hbar_i^{-1} \delta_{i j}$. Green's function yields the following solution for problems (\ref{equation:38}) and (\ref{equation:39}):
\begin{equation}
\label{equation:52}
z_{i}=\sum_{k=1}^{\mathcal{N}-1} \hbar_k \mathcal{G}^{h}\left(\xi_{i}, \zeta_{k}\right)\left(\lambda  \sum_{j=1}^{\mathcal{N}} \hbar_j \mathcal{K}_{k j} z_{j}-\mathrm{R}_{k}\right),\  \xi_{i} \in \Gamma_{\mathcal{N}}.    
\end{equation}
Similar to \cite[Theorem 1]{andreev1995convergence}, it can be shown that $0 \leq \mathcal{G}^{h}\left(\xi_{i}, \zeta_{k}\right) \leq \bar{A}^{-1}$. Therefore, we may construct the following estimate from (\ref{equation:52}):

$$
\begin{aligned}
\|z\|_{\infty, \Gamma_{\mathcal{N}}} & \leq \bar{A}^{-1} \sum_{k=1}^{\mathcal{N}-1} \hbar_k\left( |\lambda| \|z\|_{\infty, \Gamma_{\mathcal{N}}}    \sum_{j=1}^{\mathcal{N}} \hbar_j\left|\mathcal{K}_{k j}\right|+\|\mathrm{R}\|_{\infty, \Gamma_{\mathcal{N}}}\right) \\
& \leq \bar{A}^{-1} T\left(\|z\|_{\infty, \Gamma_{\mathcal{N}}}|\lambda| \widetilde{K} +\|\mathrm{R}\|_{\infty, \Gamma_{\mathcal{N}}}\right),
\end{aligned}
$$
which implies validity of (\ref{equation:51}).
\end{proof} 
\begin{theorem}
Let $\mathit{v}$ be the solution of the problem in (\ref{equation:1}) and (\ref{equation:2}), and $y$ be the solution of the discrete problem in (\ref{equation:36}) and (\ref{equation:37}), then the following $\epsilon-$uniform estimate satisfied:
\begin{equation}
\label{equation:53}
\|\mathit{v}-y\|_{\infty, \Bar{\Gamma}_\mathcal{N}} \leq C\mathcal{N}^{-2} \ln{\mathcal{N}}.
\end{equation}
\end{theorem}
\begin{proof}
Combining Lemma \ref{Lemma:2} and Lemma \ref{Lemma:3} , we immediately have (\ref{equation:53}).
\end{proof}
\section{Numerical results}
Numerical evaluations for a test problem are provided to assess the effectiveness of the numerical technique that was suggested earlier. The convergence rate and maximum pointwise error, which have been calculated, are displayed in tabular form.
\begin{table}[h!] 
    \centering
    \caption{Results for 
 Example \ref{Example1}: Convergence rates and maximum pointwise errors on $\Bar{\Gamma}_\mathcal{N}$.} 
    \label{tab:table1}
\begin{tabular}{|llllll|}
\hline
$\epsilon$  & $\mathcal{N}=64$ & $\mathcal{N}=128$ & $\mathcal{N}=256$ & $\mathcal{N}=512$ & $\mathcal{N}=1024$ \\
\hline
$2^{0}$  & $2.389e-6$ & $6.057e-7$ & $1.525e-7$ & $3.824e-8$ & $9.586e-9$ \\
 & $1.98$ & $1.99$ & $2.00$ & $2.00$ &  \\
$2^{-6}$  & $4.282e-5$ & $6.154e-7$ & $1.635e-6$ & $6.458e-7$ & $1.896e-7$ \\
 & $6.12$ & $-1.41$ & $1.34$ & $1.77$ &  \\
$2^{-12}$  & $1.012e-4$ & $2.544e-5$ & $6.348e-6$ & $1.566e-6$ & $3.750e-7$ \\
 & $1.99$ & $2.00$ & $2.02$ & $2.06$ &  \\
$2^{-18}$  & $1.0156e-4$ & $2.560e-5$ & $6.427e-6$ & $1.6099e-6$ & $4.0278e-7$ \\
 & $1.99$ & $1.99$ & $2.00$ & $2.00$ &  \\
$2^{-24}$  & $1.0157e-4$ & $2.561e-5$ & $6.428e-6$ & $1.6104e-6$ & $4.0302e-7$ \\
  & $1.99$ & $1.99$ & $2.00$ & $2.00$ &  \\
  \hline 
$e^\mathcal{N}$& $1.0157e-4$ & $2.561e-5$ & $6.428e-6$ & $1.6104e-6$ & $4.0302e-7$ \\
 $p^\mathcal{N}$ & $1.99$ & $1.99$ & $2.00$ & $2.00$ &  \\
\hline
\end{tabular}
\end{table}
\\
The maximum pointwise error is specified by:
$$E_\epsilon^\mathcal{N}=\|\mathit{v}-y\|_{\infty,\Bar{\Gamma}}$$
where $\mathit{v}$ is the exact solution and $y$ is approximate solution.  In addition, the estimates of the $\epsilon-$uniform maximum pointwise error are derived from:
$$E^\mathcal{N}=\underset{\epsilon}{\max}\  
 E_\epsilon^\mathcal{N}.$$
Convergence rates are calculated by:
$$P^\mathcal{N}_\epsilon=\frac{\ln{(E^\mathcal{N}_\epsilon/E^{2\mathcal{N}}_\epsilon})}{\ln{2}}.$$
and  $\epsilon-$uniform convergence rates are derived by:
$$P^\mathcal{N}=\frac{\ln{(E^\mathcal{N}/E^{2\mathcal{N}}})}{\ln{2}}.$$
\begin{example} \label{Example1}
\text{Consider the particular problem from} \cite{uniform_cimen_2021}:
$$\epsilon \mathit{v}''(\xi)+2\mathit{v}'(\xi)=e^\xi-\frac{1}{4} \int_0^1 e^{\xi-\eta}\mathit{v}(\eta) \mathrm{d}\eta,\ 0<\xi<1,$$
$$\mathit{v}(0)=0,\ \mathit{v}(1)=1,$$
\end{example}
in which the exact solution is $$\mathit{v}(\xi)= \frac{d_1 -1}{2+\epsilon}(1-e^\xi)+d_2 \frac{1-e^{-\frac{2\xi}{\epsilon}}}{1-e^{-\frac{2}{\epsilon}}},$$
where
\begin{align*}
d_1=&\frac{(3+\epsilon -e)(2-2e+\epsilon (1-e^{-\frac{2}{\epsilon}})) +(2+\epsilon)(e^{-\frac{2}{\epsilon}}-1)}{4e(2+\epsilon )^2 (e^{-\frac{2}{\epsilon}}-1)-(4e+\epsilon e-2e^2)+(2+\epsilon e)e^{-\frac{2}{\epsilon}}}\\
d_2=&1+\frac{(d_1-1)(e-1)}{2+\epsilon}.\\
\end{align*}
\begin{figure}[h!]
  \centering
  \includegraphics[width=1\textwidth]{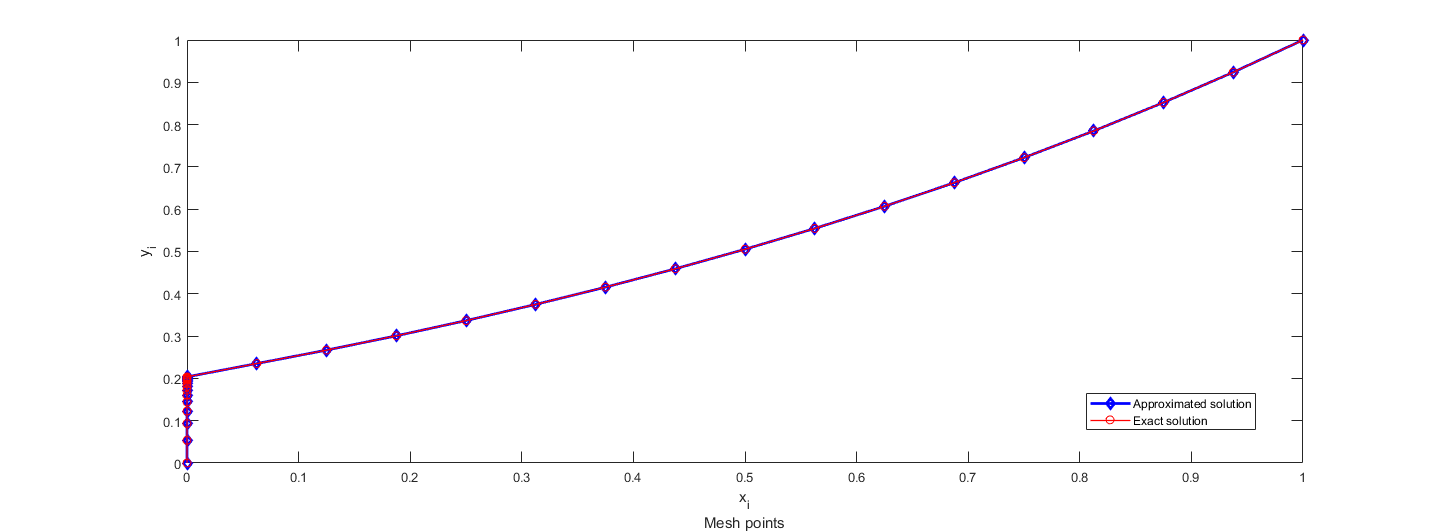}
  \caption{Solution plot of Example \ref{Example1} with $\epsilon = 2^{-24}$ and $\mathcal{N} = 128$.} 
  \label{fig:yourlabel}
\end{figure}
\begin{figure}[h!]
  \centering
  \includegraphics[width=1\textwidth]{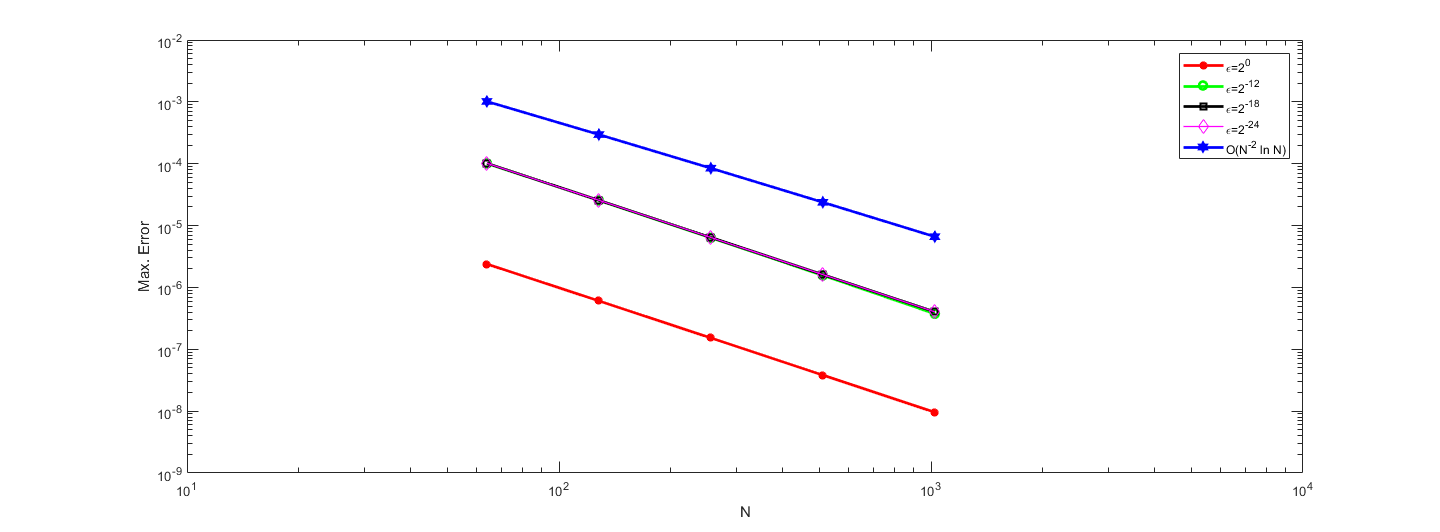}
  \caption{Loglog plots of maximum pointwise error with 
  various values of $\epsilon$ for Example \ref{Example1}}
  \label{fig:yourlabel2}
\end{figure}
Our theoretical examination demonstrates that the devised technique exhibits almost second-order uniform convergence, irrespective of 
$\epsilon$, as stated in the above section mentioned theorem. This assertion is substantiated by the numerical findings showcased in Table \ref{tab:table1} and Figures \ref{fig:yourlabel}-\ref{fig:yourlabel2}.

\section{Conclusion}
We have provided a new technique to tackle the numerical solution of a class of SPVIDEs, employing integral identities with exponential basis functions and quadrature rules. The scheme is designed on a non-uniform mesh and a thorough error analysis has been conducted, along with the resolution of a test problem. The results are shown in Figures \ref{fig:yourlabel}-\ref{fig:yourlabel2} and Tables \ref{tab:table1}, as the analysis shows, the uniform convergence rate is $O(\mathcal{N}^{-2} \ln \mathcal{N})$. These calculations affirm the stability and efficacy of the proposed method for addressing these issues. 
\nocite{*}
\bibliographystyle{acm}

\end{document}